\documentclass[11pt,reqno]{amsart}
\usepackage{amscd,graphicx,psfrag,amsxtra}

\newcommand{\p}{\partial}

\newcommand{\C}{\mathbb C}

\newcommand{\R}{\mathbb R}

\newcommand{\im}{\operatorname{im}}

\renewcommand{\phi}{\varphi}

\newcommand{\ind}{\operatorname{ind}}

\newcommand{\Hom}{\operatorname{Hom}}

\newcommand{\Ext}{\operatorname{Ext}}

\newcommand{\sign}{\operatorname{sign}}

\newtheorem{theorem}{Theorem}[section]

\newtheorem{lemma}[theorem]{Lemma}
\newtheorem{proposition}[theorem]{Proposition}

\theoremstyle{definition}

\newtheorem{remark}[theorem]{Remark}
\newtheorem{example}[theorem]{Example}

\title{Index theory of the de Rham complex on manifolds with periodic ends}

\thanks{The first author was partially supported by NSF Grant 0805841, the second author was partially supported by NSF Grant 1105234, and the third author was partially supported by NSF Grant 1065905}
\author[Tomasz Mrowka]{Tomasz Mrowka}
\address{Department of Mathematics\newline\indent Massachusetts Institute of 
Technology \newline\indent Cambridge, MA 02139}
\email{\rm{mrowka@mit.edu}}
\author[Daniel Ruberman]{Daniel Ruberman}
\address{Department of Mathematics, MS 050\newline\indent Brandeis
University \newline\indent Waltham, MA 02454}
\email{\rm{ruberman@brandeis.edu}}
\author[Nikolai Saveliev]{Nikolai Saveliev}
\address{Department of Mathematics\newline\indent
University of Miami, PO Box 249085
\newline\indent Coral Gables, FL 33124}
\email{\rm{saveliev@math.miami.edu}}

\begin{document}
\begin{abstract}
We study the de Rham complex on a smooth manifold with a periodic end modeled on an infinite cyclic cover $\tilde X \to X$. The completion of this complex in exponentially weighted $L^2$-norms is Fredholm for all but finitely many exceptional weights determined by the eigenvalues of the covering translation map $H_*(\tilde X) \to H_*(\tilde X)$. We calculate the index of this weighted de Rham complex for all weights away from the exceptional ones.
\end{abstract}
\maketitle

\section{Introduction}
Let $M$ a smooth closed orientable manifold of dimension $n$. The de Rham complex of complex valued differential forms on $M$,
\[
\begin{CD}
0 \to \Omega^0 (M) @>d_0 >> \Omega^1 (M) @>d_1 >> \Omega^2 (M) \to\; \cdots\; \to \Omega^n (M) \to 0,
\end{CD}
\]
is known to be Fredholm in a suitable $L^2$ completion. This means as usual that the images of $d_k$ are closed and the vector spaces $\ker d_k/\im d_{k-1}$ are finite dimensional. The alternating sum of the dimensions of these spaces is called the index of the de Rham complex. Since $\ker d_k/\im d_{k-1}$ is isomorphic to the singular cohomology $H^k (M;\C)$ by the de Rham theorem, the above index equals $\chi (M)$, the Euler characteristic of $M$.

This paper extends these classical results to certain non-compact manifolds, those with a periodic end modeled on an infinite cyclic cover $\tilde X$ of a closed mani\-fold $X$. It builds on the earlier work of Miller \cite{miller} and Taubes \cite{taubes:periodic} and can be viewed as a continuation of our research in \cite{mrs1} and \cite{mrs2} on the index theory of elliptic operators on such manifolds.

Let $M$ be a Riemannian manifold with periodic end modeled on $\tilde X$. Its de Rham complex can be completed in the $L^2$ norm using (on the end) a Riemannian measure $dx$ lifted from that on $X$. This completion is, however, not Fredholm; see Remark \ref{R:fred}. To rectify this problem, we will use $L^2_{\delta}$ norms, which are the $L^2$ norms on $M$ with respect to the measure $e^{\delta f(x)}\,dx$ over the end. Here, $\delta$ is a real number and $f: \tilde X \to \R$ is a smooth function such that $f(\tau(x)) = f(x) + 1$ with respect to the covering translation $\tau: \tilde X \to \tilde X$. The $L^2_{\delta}$ completion of the de Rham complex on $M$ will be denoted by $\Omega^*_{\delta} (M)$.

\begin{theorem}\label{T:one}
Let $M$ be a smooth Riemannian manifold with a periodic end modeled on $\tilde X$, and suppose that $H_* (M;\C)$ is finite dimensional. Then $\Omega^*_{\delta} (M)$ is Fredholm for all but finitely many $\delta$ of the form $\delta = \ln|\lambda|$, where $\lambda$ is a root of the characteristic polynomial of $\tau_*: H_* (\tilde X;\C) \to H_* (\tilde X;\C)$.
\end{theorem}

Conditions on $X$ that guarantee that $H_* (M;\C)$ is finite dimensional can be found in Section \ref{S:fred}, together with a proof of Theorem \ref{T:one}. 

Given a manifold $M$ as in the above theorem, the complex $\Omega^*_{\delta} (M)$ has a well defined index $\ind_{\delta} (M)$. Miller \cite{miller} showed that $\ind_{\delta} (M)$ is an even or odd function of $\delta$ according to whether $\dim M = n$ is even or odd, and that $\ind_{\delta} (M) = (-1)^n\,\chi (M)$ for sufficiently large $\delta > 0$. We add to this knowledge the following result.

\begin{theorem}\label{T:two} 
Let $M$ be as in Theorem \ref{T:one}. Then $\ind_{\delta} (M)$ is a piecewise constant function of $\delta$ whose only jumps occur at $\delta = \ln |\lambda|$, where $\lambda$ a root of the characteristic polynomial $A_k(t)$ of $\tau_*: H_k (\tilde X;\C) \to H_k (\tilde X;\C)$ for some $k = 0,\ldots,n-1$. Every such $\lambda$ contributes $(-1)^{k+1}$ times its multiplicity as the root of $A_k(t)$ to the jump.
\end{theorem}

Together with the results of \cite{miller} this completes the calculation of the function $\ind_{\delta} (M)$. Theorem \ref{T:two} is proved in Section \ref{S:index}. The last section of the paper contains discussion as well as calculations of $\ind_{\delta} (M)$ for two important classes of examples. The first class consists of manifolds with infinite cylindrical ends studied earlier by Atiyah, Patodi and Singer \cite{aps:I}, and the second of manifolds arising in the study of knotted $S^2 \subset S^4$.

Finally, we will remark that the de Rham complex is a special case of the more general concept of an elliptic complex. The index theory for elliptic complexes on closed manifolds was developed by Atiyah and Singer whose famous index theorem \cite{atiyah-singer:I} expresses the index of such a complex in purely topological terms. More generally, Atiyah, Patodi and Singer \cite{aps:I} computed the index for certain elliptic operators (that is, elliptic complexes of length two) on manifolds with cylindrical ends. In our paper~\cite{mrs2} we extended their result to general manifolds with periodic ends; our index formula involves a new end-periodic $\eta$-invariant, generalizing the $\eta$-invariant of Atiyah, Patodi and Singer from the cylindrical setting. It would be interesting to compare the index formula of Theorem \ref{T:two} with that of \cite{mrs2} for the operator $d + d^*$ obtained by wrapping up the de Rham complex. 

\medskip\noindent
\textbf{Acknowledgments:} We thank Dan Burghelea and Andrei Pajitnov for  enlightening discussions on this material.


\section{The Fredholm property}\label{S:fred}
In this section, we will prove Theorem \ref{T:one} by reducing it to a statement about twisted cohomology of $X$. 


\subsection{The Fourier--Laplace transform}
The de Rham complex of $M$ is an elliptic complex on the end-periodic manifold $M$ hence we can use the general theory of such complexes due to Taubes \cite{taubes:periodic}. According to that theory, it is sufficient to check the Fredholm property of the de Rham complex of $\tilde X$ completed in the $L^2$ norm on $\tilde X$ with respect to the measure $e^{\delta f(x)}\,dx$. The latter complex can be studied using the Fourier--Laplace transform, which is defined by the formula
\[
\hat\omega_z \;=\;\sum\;z^k\cdot (\tau^*)^k\;\omega,\quad z \in \C^*,
\]
on compactly supported forms $\omega$ on $\tilde X$, and is extended by continuity to $L^2$ forms. The summation in the above formula extends to all integer $k$, which makes the form $\hat\omega_z$ invariant with respect to $\tau^*$. The form $\hat\omega_z$ then defines a form on $X$ which is denoted by the same symbol. An application of the Fourier--Laplace transform to the de Rham complex on $\tilde X$ results in a family of twisted de Rham complexes 
\begin{equation}\label{E:twist}
\begin{CD}
\cdots @>>> \Omega^k (X) @> \quad d - \ln z\,df\quad >> \Omega^{k+1} (X) @>>> \ldots
\end{CD}
\end{equation}
parameterized by $z \in \C^*$. The following result is proved in Taubes \cite[Lemma 4.3]{taubes:periodic}.

\begin{proposition}\label{P:taubes}
Let $M$ be a smooth Riemannian manifold with  a periodic end modeled on $\tilde X$.  For any given $\delta$, the complex $\Omega^*_{\delta} (M)$ is Fredholm if and only if the complexes \eqref{E:twist} are exact for all $z$ such that $|z| = e^{\delta}$.
\end{proposition}

The cohomology of complex \eqref{E:twist} is of course the twisted de Rham cohomology $H^*_z (X;\C)$ with coefficients in the complex line bundle with flat connection $-\ln z\,df$. 

\begin{proposition}\label{P:fred}
Let $M$ be a smooth Riemannian manifold with  a periodic end modeled on $\tilde X$. Then the following three conditions are equivalent\,:
\begin{enumerate}
\item $\Omega^*_{\delta} (M)$ is Fredholm for all $\delta \in \R$ away from a  discrete set;
\item $H^*_z (X;\C)$ vanishes for all $z\in \C^*$ away from a  discrete set;
\item $H^*_z (X;\C)$ vanishes for at least one $z \in \C^*$.
\end{enumerate}
\end{proposition}

\begin{proof}
Observe that the complexes \eqref{E:twist} form a holomorphic family of elliptic complexes on $\C^*$; therefore, exactness of \eqref{E:twist} at one point is equivalent to exactness away from a discrete set.
The statement now follows from the preceding discussion. 
\end{proof}

\begin{remark}\label{R:fred}
The usual $L^2$ completion of the de Rham complex on $M$, that is, the complex $\Omega^*_0 (M)$, is not Fredholm because $H^0_z (X;\C)$ is not zero when $z = 1$.
\end{remark}


\subsection{Finite dimensionality}
Fix a finite cell complex structure on $X$, lift it to $\tilde X$, and consider the chain complex $C_* (\tilde X,\C)$ and its homology $H_* (\tilde X;\C)$. The group of integers acts on both by covering translations making them into finitely generated modules over $\C[t,t^{-1}]$. The twisted de Rham theorem tells us that the cohomology $H^*_z (X;\C)$ of complex \eqref{E:twist} is isomorphic to the cohomology of the complex $\Hom_{\,\C[t,t^{-1}]}\;(C_* (\tilde X,\C), \C_z)$, where $\C_z$ denotes a copy of $\C$ viewed as a $\C[t,t^{-1}]$ module with $p(t)$ acting via multiplication by $p(z)$. 

\begin{proposition}\label{P:finite}
Let $M$ be a smooth Riemannian manifold with  a periodic end modeled on $\tilde X$. Then the following two conditions are equivalent\,:
\begin{enumerate}
\item $H_*(M;\C)$ is a finite dimensional vector space;
\item $H^*_z (X;\C)$ vanishes for at least one $z \in \C^*$.
\end{enumerate}
\end{proposition}

\begin{proof}
It is immediate from the Mayer--Vietoris principle that $H^*(M;\C)$ is finite dimensional if and only if $H^*(\tilde X;\C)$ is finite dimensional. Since $H_*(\tilde X;\C)$ is a finitely generated module over the principal ideal domain $\C[t,t^{-1}]$, it admits a primary decomposition 
\begin{equation}\label{E:prime}
H_*(\tilde X;\C)\,=\,\C[t,t^{-1}]^{\ell}\,\oplus\,\C[t,t^{-1}]/(p_1)\,\oplus\ldots\oplus\,\C[t,t^{-1}]/(p_m),
\end{equation}
therefore, $H^*(\tilde X;\C)$ is a finite dimensional vector space if and only if $\ell = 0$ in this decomposition. According to the universal coefficient theorem,
\[
H^*_z (X;\C)\;=\;\Hom_{\,\C[t,t^{-1}]}\;(H_* (\tilde X;\C),\C_z)\,\oplus\,\Ext_{\,\C[t,t^{-1}]}\;(H_* (\tilde X;\C),\C_z),
\]
hence vanishing of $H^*_z (X;\C)$ for at least one $z$ implies that $\ell = 0$. On the other hand, an easy calculation shows that 
\[
\Hom_{\,\C[t,t^{-1}]}\;(V,\C_z)\;=\;\Ext_{\,\C[t,t^{-1}]}\;(V,\C_z)\;=\;0
\] 
for any module $V = \C[t,t^{-1}]/(p)$ such that $p(z) \neq 0$. Therefore, $\ell = 0$ implies that $H^*_z (X;\C)$ must vanish for all $z$ away from the roots of the polynomials $p_1,\ldots,p_m$.
\end{proof}


\subsection{Proof of Theorem \ref{T:one}}
It follows from Propositions \ref{P:fred} and \ref{P:finite} that, if $H^* (\tilde X;\C)$ is finite dimensional, the complex $\Omega^*_{\delta} (M)$ is Fredholm for all $\delta$ away from a  discrete set. To finish the proof of Theorem \ref{T:one} we just need to identify this discrete set. According to Proposition \ref{P:taubes}, it consists of $\delta = \ln |z|$, where $z \in \C^*$ are the complex numbers for which $H^*_z (X;\C)$ fails to be zero. To find them, note that the free part in the prime decomposition \eqref{E:prime} vanishes making $H_*(\tilde X;\C)$ into a torsion module,
\begin{equation}\label{E:torsion}
H_*(\tilde X;\C)\,=\,\C[t,t^{-1}]/(p_1)\,\oplus\ldots\oplus\,\C[t,t^{-1}]/(p_m).
\end{equation}
According to Milnor \cite[Assertion 4]{milnor}, the order ideal $(p_1\cdots p_m)$ of this module is spanned by the characteristic polynomial of $\tau_*: H_* (\tilde X;\C) \to H_* (\tilde X;\C)$. The calculation with the universal coefficient theorem as in the proof of Proposition \ref{P:finite} now completes the proof.


\subsection{A sufficient condition}
Let $M$ be a smooth orientable manifold with a periodic end modeled on $\tilde X$. Vanishing of $\chi (X)$ is obviously a necessary condition for the vector space $H_* (M;\C)$ to be finite dimensional. To come up with a sufficient condition, observe that the derivative $df$ defines a closed 1-form on $X$, and let $\xi = [df] \in H^1 (X;\C)$ be its cohomology class. The cup product with $\xi$ gives rise to the chain complex
\begin{equation}\label{E:xi}
H^0 (X;\C) \xrightarrow{\;\cup\,\xi\;} H^1 (X;\C) \xrightarrow{\;\cup\,\xi\;}\;\ldots\;\xrightarrow{\;\cup\,\xi\;} H^n (X;\C).
\end{equation}

\begin{proposition}
Suppose the chain complex \eqref{E:xi} is exact. Then $H_*(M;\C)$ is a finite dimensional vector space for any smooth orientable manifold with periodic end modeled on $\tilde X$.
\end{proposition}

This proposition can be derived as a special case of Taubes \cite[Theorem 3.1]{taubes:periodic}. That proof is analytic in nature; here is another proof which is purely topological.

According to Proposition \ref{P:finite}, it is sufficient to prove that the twisted cohomology $H^*_z (X;\C)$ vanishes for at least one $z \in \C^*$. We will show that it does so for all $z \neq 1$ in a sufficiently small neighborhood of 1. For such $z$, there is a spectral sequence $(E^*_r,d_r)$ which starts at $E^*_1 = H^*(X;\C)$, converges to $H^*_z (X;\C)$, and whose differentials are given by the Massey products with the class $\xi$; see Farber \cite[Section 10.9]{farber}. The convergence of this spectral sequence to zero is therefore a necessary and sufficient condition for the finite dimensionality of $H_*(M;\C)$. The chain complex \eqref{E:xi} is the term  $(E^*_1,d_1)$ of that spectral sequence, hence its exactness is sufficient for the finite dimensionality of $H^*(X;\C)$.

Note that the vanishing of the Euler characteristic of $X$ is not a sufficient condition for $H_*(M;\C)$ to be finite dimensional. An example is provided by the connected sum of $S^1 \times S^{n-1}$ with any manifold that is not a rational homology sphere but has Euler characteristic $2$.


\section{The index calculation}\label{S:index}
Let $M$  be a smooth Riemannian manifold with periodic end modeled on the infinite cyclic cover $\tilde X$, and assume that $H_* (M;\C)$ is finite dimensional. Let $\tau: \tilde X \to \tilde X$ be a covering translation, and denote by $A_k (t)$ the characteristic polynomial of $\tau_*: H_k(\tilde X;\C) \to H_k(\tilde X;\C)$ (the polynomial $A_1 (t)$ is traditionally referred to as the Alexander polynomial of the fundamental group of $X$). Denote by $\Delta$ the set of all $\delta$ of the form $\delta = \ln|\lambda|$, where $\lambda$ is a root of the product polynomial $A_0 (t)\cdot\ldots\cdot A_{n-1}(t)$. 

According to Theorem \ref{T:one}, the complex $\Omega^*_{\delta} (M)$ is Fredholm for all $\delta$ away from $\Delta$. Its index $\ind_{\delta}(M)$ is a piecewise constant function away from $\Delta$, where it may jump. We wish to calculate the size of these jumps.


\subsection{Excision principle}\label{S:excision}
Let $\delta_1$ and $\delta_2$ be two weights in $\R - \Delta$ and complete the de Rham complex of $\tilde X$ in the $L^2$ norm with respect to the measure $e^{\delta_1 f(x)}\,dx$ on the negative end of $\tilde X$, and $e^{\delta_2 f(x)}\,dx$ on the positive end. This complex will be denoted by $\Omega^*_{\delta_1\delta_2} (\tilde X)$. This is a Fredholm complex whose index will be denoted by $\ind_{\delta_1\delta_2}(\tilde X)$. 

\begin{proposition}\label{P:change}\quad
$\ind_{\delta_2}(M) - \ind_{\delta_1}(M)\;=\;\ind_{\delta_1\delta_2}(\tilde X)$.
\end{proposition}

\begin{proof} 
Let $c\in \R$ be a regular value of $f: \tilde X \to \R$ then $Y = f^{-1} (c)$ is a submanifold of $\tilde X$ separating it as $\tilde X = \tilde X_-\,\cup\,\tilde X_+$. Write $M = Z\,\cup\,\tilde X_+$ for some smooth compact manifold $Z$ with boundary $Y$. An application of the exicision principle to these two splittings yields 
\[
\ind_{\delta_1} (M)\,+\,\ind_{\delta_1\delta_2} (\tilde X)\,=\,\ind_{\delta_2} (M)\,+\,\ind_{\delta_1\delta_1} (\tilde X).
\]
Note that the complex $\Omega^*_{\delta_1\delta_1}(\tilde X)$ is exact because the complexes \eqref{E:twist} obtained from it by the Fourier--Laplace transform are exact for all $z$ with $|z| = e^{\delta_1}$. Therefore, $\ind_{\delta_1\delta_1} (\tilde X) = 0$ and the proof is complete.
\end{proof}


\subsection{Computing cohomology of $\Omega^*_{\delta_1\delta_2}(\tilde X$)} 
We will proceed by several reductions, the first being from weighted forms to weighted cellular cochains. To be precise, fix a finite cell complex structure on $X$ and lift it to $\tilde X$. Also, introduce the Hilbert space $\ell^2_{\delta_1\delta_2}$ of the sequences $\{x_k\;|\; k\in \mathbb Z\}$ of complex numbers such that 
\smallskip\[
\sum_{k < 0}\; e^{2\delta_1 k}\,|x_k|^2\; < \;\infty\qquad\text{and}
\qquad \sum_{k > 0}\; e^{2\delta_2 k}\,|x_k|^2\; < \;\infty.
\]

\smallskip\noindent
Theorem 2.17 of Miller \cite{miller}, which is a weighted version of the $L^2$ de Rham theorem, see \cite{dodziuk,lueck:L2}, establishes an isomorphism between the cohomology of $\Omega^*_{\delta_1\delta_2} (\tilde X)$ and the cellular cohomology of $\tilde X$ with $\ell^2_{\delta_1\delta_2}$ coefficients.  Miller actually uses weighted {\em simplicial} cohomology, but this is readily seen to be isomorphic to the more standard and convenient cellular version. 

\begin{proposition}\label{P:shift}
View $\ell^2_{\delta_1\delta_2}$ as a $\C[t,t^{-1}]$ module with $t$ acting as the right shift operator, $t(x_k) = x_{k+1}$. Then for all but finitely many $\delta_1$ and $\delta_2$, the cohomology of $\tilde X$ with $\ell^2_{\delta_1\delta_2}$ coefficients equals the homology of the complex
\begin{equation}\label{E:hom-l2}
\Hom_{\,\C[t,t^{-1}]}\;(C_*(\tilde X,\C),\,\ell^2_{\delta_1\delta_2}).
\end{equation}
\end{proposition}

\begin{proof}
This will follow as soon as we show that the images of the boundary operators $\p$ in complex \eqref{E:hom-l2} are closed. These boundary operators
\[
\p: \left(\ell^2_{\delta_1 \delta_2}\right)^k \to \left(\ell^2_{\delta_1 \delta_2}\right)^{\ell}
\]
are matrices whose entries are Laurent polynomials in $t$. Since $\C[t,t^{-1}]$ is a principal ideal domain, each $\p$ will have a diagonal matrix in properly chosen bases. The statement now follows from the fact that the operator $t - \lambda: \ell^2_{\delta_1 \delta_2} \to \ell^2_{\delta_1 \delta_2}$ is Fredholm for all $\lambda$ with $|\lambda|$ different from $e^{\delta_1}$ and $e^{\delta_2}$; see for instance Conway \cite[Proposition 27.7 (c)]{conway}.
\end{proof}

The universal coefficient theorem now tells us that the $\ell^2_{\delta_1 \delta_2}$ cohomology of $\tilde X$ is isomorphic to 
\begin{equation}\label{E:coeff}
\Hom_{\,\C[t,t^{-1}]}\;(H_*(\tilde X;\C),\ell^2_{\delta_1 \delta_2})\,\oplus\,
\Ext_{\,\C[t,t^{-1}]}\;(H_*(\tilde X;\C),\ell^2_{\delta_1 \delta_2}).
\end{equation}
Recall from the proof of Theorem \ref{T:one} that $H_* (\tilde X;\C)$ is a torsion module \eqref{E:torsion} whose order ideal is spanned by the characteristic polynomial of $\tau_*: H_* (\tilde X;\C) \to H_* (\tilde X;\C)$. Therefore, our next step will be to compute \eqref{E:coeff}, one cyclic module at a time.

\begin{lemma}
Let $\lambda$ be a complex number such that $|\lambda|$ is different from $e^{\delta_1}$ and $e^{\delta_2}$. Then, for any cyclic module $V = \C[t,t^{-1}]/(t-\lambda)^m$, we have
\[
\Ext_{\,\C[t,t^{-1}]}\;(V,\ell^2_{\delta_1 \delta_2})\,=\,0.
\]
\end{lemma}

\begin{proof}
We already know from the proof of Proposition \ref{P:shift} that the operator $t - \lambda: \ell^2_{\delta_1 \delta_2} \to \ell^2_{\delta_1 \delta_2}$ is Fredholm. In addition, one can easily check that all finite sequences belong to its image. Since such sequences are dense in $\ell^2_{\delta_1 \delta_2}$, the operator $t - \lambda$ is surjective, and so are the operators $(t - \lambda)^m$ for all $m$. The result is now immediate from the definition of $\Ext$.
\end{proof}

\begin{lemma}
Let $\lambda$ be a complex number such that $|\lambda|$ is different from $e^{\delta_1}$ and $e^{\delta_2}$. Assume that $\delta_2 < \delta_1$. Then, for any cyclic module $V = \C[t,t^{-1}]/(t-\lambda)^m$, the dimension of $\Hom_{\,\C[t,t^{-1}]}\;(V,\ell^2_{\delta_1 \delta_2})$ is $m$ if $e^{\delta_2} < |\lambda| < e^{\delta_1}$, and zero otherwise.
\end{lemma}

\begin{proof}
For such a module, $\Hom_{\,\C[t,t^{-1}]}\,(V,\ell^2_{\delta_1 \delta_2})$ equals the kernel of the operator $(t - \lambda)^m: \ell^2_{\delta_1 \delta_2} \to \ell^2_{\delta_1 \delta_2}$. Computing this kernel is a straightforward exercise with infinite series.
\end{proof}


\subsection{Proof of Theorem \ref{T:two}}
Let $\lambda$ be a root of the product polynomial $A_0(t)\cdots A_{n-1}(t)$ of multiplicity $m = m_0 + \ldots + m_{n-1}$, where $m_k$ is the multiplicity of $\lambda$ as a root of $A_k(t)$. Choose generic $\delta_1$ and $\delta_2$ so that $e^{\delta_2} < |\lambda| < e^{\delta_1}$ and there are no other roots of $A_0(t)\cdots A_{n-1}(t)$ whose absolute values fit in this interval. It follows from Proposition \ref{P:change} and the cohomology calculation in the previous section that
\[
\ind_{\delta_1} (M)\;=\;\ind_{\delta_2} (M)\;-\;\sum\;(-1)^k\,m_k,
\]
which is exactly the formula claimed in Theorem \ref{T:two}.


\section{Discussion and examples}
Let $M$ be a smooth Riemannian manifold of dimension $n$ with a periodic end modeled on $\tilde X$ and suppose that $H^*(\tilde X;\C)$ is finite dimensional. Then, for any $\delta \in \R - \Delta$, the de Rham complex $\Omega^*_{\delta} (M)$ is Fredholm and its index is given by the formula
\[
\ind_{\delta} (M)\;=\; (-1)^n\,\chi(M)\; + \;\sum\;(-1)^k\;\# \{\lambda\,|\, A_k(\lambda)=0,\;|\lambda| > e^{\delta}\,\},
\]
where the roots $\lambda$ of $A_k (t)$ are counted with their multiplicities. This formula is obtained by combining Theorem \ref{T:two} with Miller's theorem \cite{miller} that $\ind_{\delta} (M) = (-1)^n\,\chi (M)$ for sufficiently large $\delta > 0$. Miller \cite{miller} also shows that the function $\ind_{\delta} (M)$ is even or odd depending on whether $n$ is even or odd. This is consistent with the above formula because of Blanchfield duality which says that $A_k(\lambda) = 0$ if and only if $A_{n-k-1} (1/\lambda) = 0$ with matching multiplicities.

\begin{example}
A manifold with product end is a smooth Riemannian manifold whose end is modeled on $\tilde X = \mathbb R \times Y$, where $Y$ is a closed Riemannian manifold. The metric on $\R \times Y$ is presumed to be the product metric. The index theory on such manifolds has been studied by Atiyah, Patodi and Singer \cite{aps:I}. The covering translation induces an identity map $\tau_*$ on the homology of $\R \times Y$. Since $\lambda = 1$ is the only root of the characterictic polynomial of $\tau_*$ the complex $\Omega^*_{\delta}\,(M)$ is Fredholm for all $\delta \ne 0$. Its index $\ind_{\delta} (M)$ equals $\chi (M)$ if the dimension of $M$ is even, and $\sign(-\delta)\cdot \chi(M)$ if the dimension of $M$ is odd. Note that the same is true for any manifold whose periodic end is modeled on $\tilde X$ such that the characteristic polynomial of $\tau_*: H_* (\tilde X;\C) \to H_* (\tilde X;\C)$ only has unitary roots.
\end{example}

\begin{example}
This example originates in Fox's `Quick Trip'~\cite[Example 11]{fox:trip}. Fox constructs a 2-knot in the 4-sphere with with the property that the infinite cyclic cover of its exterior has first homology isomorphic to the additive group of dyadic rationals. A nice plumbing construction of this knot described in Rolfsen's book \cite[Section 7.F]{rolfsen:knots} shows that it has a Seifert surface diffeomorphic to $S^1 \times S^2 - D^3$. Perform a surgery on the knot so that the Seifert surface is capped off by the core 3-disk of the surgery. The resulting manifold $X$ has integral homology of $S^1 \times S^3$. It follows from a calculation in \cite[Section 7.F]{rolfsen:knots} that the characteristic polynomials $A_k (t)$ of the covering translation $\tau_*: H_k (\tilde X;\C) \to H_k (\tilde X;\C)$ are as follows: $A_0 (t) = t - 1$, $A_1 (t) = t - 2$, $A_2 (t) = t - 1/2$, and $A_3 (t) = t - 1$. Cut $\tilde X$ along a copy of $S^1 \times S^2$ and fill it in by $D^2 \times S^2$ to obtain an end-periodic manifold $M$. A straightforward calculation shows that $\chi (M) = 2$. The complex $\Omega^*_{\delta}(M)$ is Fredholm away from $\delta = 0$ and $\delta = \pm \ln 2$. Its index is equal to 1 if $0 < |\delta| < \ln 2$, and is equal to 2 otherwise.
\end{example}

\medskip



\begin{thebibliography}{10} 

\bibitem{aps:I}
M.~Atiyah, V.~Patodi, and I.~Singer, {\em Spectral asymmetry and {Riemannian}
 geometry: {I}}, Math. Proc. Camb. Phil. Soc., {\bf 77} (1975), 43--69

\bibitem{atiyah-singer:I}
M.~Atiyah and I.~Singer, {\em The index of elliptic operators: {I}}, Annals of
  Math., {\bf 87} (1968), 484--530.

\bibitem{conway}
J.~Conway, 
A course in operator theory. Amer. Math. Society, Providence, 2000

\bibitem{dodziuk}
J.~Dodziuk,
\emph{De Rham-Hodge theory for $L^2$-cohomology of infinite coverings}, 
Topology \textbf{16} (1977), 157--165

\bibitem{farber}
M.~Farber,
Topology of closed one-forms. Amer. Math. Society, Providence, 2004

\bibitem{fox:trip}
R.~H. Fox,  {\em A quick trip through
  knot theory}, in ``Topology of 3-manifolds and related topics ({P}roc. {T}he
  {U}niv. of {G}eorgia {I}nstitute, 1961)'', Prentice-Hall, Englewood Cliffs,
  N.J., 1962, 120--167.

\bibitem{lueck:L2}
W.~L{\"u}ck, ``{$L^2$}-invariants: theory and applications to geometry and
  {$K$}-theory'', vol.~44 of Ergebnisse der Mathematik und ihrer Grenzgebiete.
  3. Folge. A Series of Modern Surveys in Mathematics [Results in Mathematics
  and Related Areas. 3rd Series. A Series of Modern Surveys in Mathematics],
  Springer-Verlag, Berlin, 2002.

\bibitem{miller}
J.~Miller,
\emph{The Euler characteristic and finiteness obstruction of manifolds with periodic ends}, Asian J. Math. \textbf{10} (2006), 679--714

\bibitem{milnor}
J.~Milnor,
\emph{Infinite cyclic coverings}. 1968 Conference on the Topology of Manifolds (Michigan State Univ., E. Lansing, Mich., 1967) pp. 115--133 Prindle, Weber \& Schmidt, Boston.

\bibitem{mrs1}
T.~Mrowka, D.~Ruberman, N.~Saveliev,
\emph{Seiberg--Witten equations, end-periodic Dirac operators, and a lift of Rohlin's invariant}, J. Differential Geom. \textbf{88} (2011), 333--377

\bibitem{mrs2}
T.~Mrowka, D.~Ruberman, N.~Saveliev,
\emph{An index theorem for end-periodic operators}. Preprint arXiv:1105.0260

\bibitem{pajitnov}
A.~Pajitnov,
\emph{Proof of a conjecture of Novikov on homology with local coefficients over a field of finite characteristic}. (Russian) Dokl. Akad. Nauk SSSR \textbf{300} (1988), 1316--1320; translation in Soviet Math. Dokl. \textbf{37} (1988), 824--828

\bibitem{rolfsen:knots}
D.~Rolfsen, Knots and Links. Publish or Perish, Berkeley, 1976.

\bibitem{taubes:periodic} 
C.~Taubes, 
\emph{Gauge theory on asymptotically periodic $4$-manifolds}, 
J. Differential Geom. \textbf{25} (1987), 363--430 

\end{thebibliography}
\end{document}